\documentclass{amsart}

\usepackage{color}
\usepackage{amsthm,amssymb,verbatim}
\usepackage{graphicx}
\usepackage{enumerate}
 \newcommand{\MM}{\mathcal{M}}
 
 \newcommand{\Ss}{\mathcal{S}}
 \newcommand{\RR}{\mathbf{R}}  % reals
 \newcommand{\eps}{\epsilon}

\input epsf
\def\begfig {
\begin{figure}[b]
\small }
\def\endfig {
\normalsize
\end{figure}
}

\swapnumbers

    \newtheorem{theorem}    {Theorem}       [section]
    \newtheorem{lemma}      [theorem]       {Lemma}
    \newtheorem{corollary}  [theorem]     {Corollary}

    \newtheorem*{theorem*}{Theorem}
    \newtheorem*{corollary*}{Corollary}

    \theoremstyle{definition}
    \newtheorem{definition}  [theorem] {Definition}

    \theoremstyle{definition}
    \newtheorem{remark}   [theorem]       {Remark}
    \newtheorem{conjecture} [theorem]  {Conjecture}

\usepackage{hyperref} 
\usepackage[alphabetic, msc-links, backrefs]{amsrefs}

\begin{document}

\renewcommand{\thesubsection}{\thetheorem}
   % this is so subsections and theorems, etc will be
   % numbered together.

\title[Bumpy Metrics]{On the Bumpy Metrics Theorem for Minimal Submanifolds}
\author{Brian White}
\address{Department of Mathematics\\ Stanford University\\ Stanford, CA 94305}
\thanks{The research was supported by NSF grants~DMS--1105330 and DMS--1404282.}
\email{white@math.stanford.edu}
\date{March 5, 2015.}
\begin{abstract}
This paper proves several natural generalizations of the theorem
that for a generic, $C^k$ Riemannian metric on a smooth manifold,
there are no closed, embedded, minimal submanifolds with nontrivial jacobi fields.
\end{abstract}

\maketitle

\stepcounter{theorem}
%\subsection{Bumpiness}

\newcommand{\stab}{\operatorname{stab}}

\begin{comment}
\begin{definition}
Note $N$ be a smooth manifold.
For $i=1,2$, suppose that $\Sigma_i$ is a smooth manifold
with smooth (possibly empty) boundary, and that
\[
  \iota_i: \Sigma_i\to N
\]
is a $C^1$ immersion.  We say that $\sigma_1$ is equivalent
by reparametrization to $\sigma_2$ if $\iota_1=\iota_2\circ u$
for some bijection diffeomorphism $u:\Sigma_1\to \Sigma_2$.
If $\iota$ is such an immersion, we let $[\iota]$ denote its
equivalence class.
\end{definition}

Let $\Sigma_i$ and $\Sigma'$ be smooth, connected manifolds, possibly
with smooth boundary.   Two immersions $\iota:\Sigma\to$
Let $\Sigma$ be a smooth $m$-manifold, possibly with smooth boundary.
If $\iota: \Sigma\to N$ is an immersion, we let $[\iota]$ be the
equivalence class of such immersions, where $\iota$
\end{comment}

%In this section, ``manifold" means ``manifold without boundary".

\section{Introduction}

According to the bumpy metrics theorem of~\cite{white-space}*{2.2},
 for a generic (in the sense of Baire Category) $C^k$ Riemannian metric
on a manifold $N$, there are no closed, embedded, minimal submanifolds with nontrivial
jacobi fields.   (The case of geodesics, including immersed geodesics, was proved earlier in~\cite{abraham-bumpy}).
That theorem leaves open the following questions:
\begin{enumerate}
\item Is the theorem true with $C^\infty$ in place of $C^k$?
\item Is the theorem true for immersed submanifolds?
\item Suppose $G$ is a finite group of diffeomorphisms of $N$.
The proofs in~\cite{white-space} work equally well in the $G$-invariant setting:
for a generic, $G$-invariant, $C^k$ Riemannian metric on $N$, there
are no closed, embedded  $G$-invariant minimal surfaces with 
nontrivial $G$-invariant jacobi fields.   
But is it also true that for a generic $G$-invariant
metric, there are no closed minimal surfaces with nontrivial jacobi fields?
\end{enumerate}

In this paper, we prove that the answer to each of the three questions is ``yes".

The proofs in this paper, which show that these results follow from results in~\cite{white-space},
are fairly simple. (The reader need not be familiar with that paper, provided he or she is willing
to assume the results from that paper.)  However, the issues involved are in some ways subtle.
In particular, one can ask the same questions 
about minimal submanifolds with generic boundaries (the metric being fixed); 
in that setting, the analog of question (3)
is open, even for two dimensional minimal surfaces in $\RR^3$.
See Section~\ref{generalizations-section}.

The reader may wonder why the methods of proof in~\cite{white-space} do not immediately
give the answers to (1), (2), and (3).
The bumpy metrics theorem there is proved as follows.  The space of pairs $(\gamma,M)$
where $\gamma$ is a $C^k$ Riemannian metric on the ambient manifold and where $M$
is a $\gamma$-minimal, embedded submanifold is proved to be a separable Banach
manifold, and the ``bad'' metrics $\gamma$ are precisely the critical values of the
map $(\gamma,M)\mapsto \gamma$, so the bumpy metrics theorem follows immediately
from the Sard-Smale Theorem~\cite{smale-sard}. To handle $C^\infty$ metrics directly,
one would have to use Frechet manifolds rather than Banach manifolds. Presumably
that would work, but we have chosen instead to prove directly that the result for $C^k$ implies the
result for $C^\infty$.  See Theorem~\ref{to-infinity-theorem}.

Concerning question (2), it would be natural to consider the space of pairs $(\gamma,M)$ where $\gamma$ is a $C^k$
metric and $M$ is an {\em immersed}, $\gamma$-minimal manifold.  However, that space
is, in general, {\em not} a Banach manifold.  (It fails to be a manifold at those points $(\gamma,M)$ where $M$ is a multiple
cover of another minimal surface $\Sigma$.)  
Presumably, one could show that the space is a nicely stratified Banach variety and use that to answer question (2),
 but I suspect that would be more complicated than the proof given here.

The same difficulty arises in connection with question (3).
One could consider the space of pairs $(\gamma,M)$ where the metric
$\gamma$ is invariant with respect to a specified finite group $G$ and
where $M$ is a closed, $\gamma$-minimal (but not necessarily $G$-invariant)
submanifold.  Again, in general that space will not be a Banach manifold.
(The problematic points are the pairs $(\gamma,M)$ where $M$ is $G$-invariant
but has a nontrivial jacobi field that is not $G$-invariant.)

\section{The main theorem}

The following is the main result in this paper (see section~\ref{generalizations-section} for some generalizations):

\begin{theorem}\label{main-theorem}
Let $N$ be a smooth manifold.
Let $G$ be a finite group of diffeomorphisms of $N$.
Suppose that $k$ is an integer $\ge 3$ or that $k=\infty$. 
Then a generic,  $G$-invariant, $C^k$ Riemannian metric on $N$
is bumpy in the following sense:
 no closed, minimal immersed submanifold $M$
of $N$ has a nontrivial jacobi field.
\end{theorem}

Note that the minimal submanifold $M$ is not required to be $G$-invariant,
and that even if it is $G$-invariant, the jacobi fields referred to need not be
$G$-invariant.

We will prove a series of results that together imply Theorem~\ref{main-theorem}.

\begin{definition}
An ``almost embedded'' surface is an immersed surface that has multiplicity one on an open dense
subset.
\end{definition}

\begin{theorem}[Structure Theorem]\label{structure-theorem}
Let $k$ be an integer $\ge 3$.
Let $\Gamma=\Gamma_k$ be the space of $G$-invariant, $C^k$ Riemannian
metrics on a smooth manifold $N$.  
Let $\MM=\MM_k$ be the space of pairs $(\gamma, M)$ where $\gamma\in \Gamma_k$
and $M$ is a closed, almost embedded, $G$-invariant, 
$\gamma$-minimal surface in $N$.   Then $\MM$ is a separable, $C^{k-2}$ Banach
manifold, and the map
\begin{align*}
&\pi: \MM\to \Gamma  \\
&\pi(\gamma,M)= \gamma
\end{align*}
is a $C^{k-2}$ Fredholm map of Fredholm index $0$.   Furthermore, $(\gamma,M)$ is 
a critical point of $\pi$ if and only if $M$ has a nontrivial, $G$-equivariant jacobi field.
\end{theorem}

\begin{proof}
If we fix the diffeomorphism type of $M$, if $G$ is the trivial group, and if we replace 
``almost embedded" by ``embedded",
this is Theorem~2.1  of~\cite{white-space}.  
See section~7 of that paper for the extension to almost embedded
surfaces.  (In that paper, ``almost embedded surfaces"
are called ``simply immersed surfaces".) 
For general groups $G$, exactly the same proofs work
provided one replaces ``metric", ``surface", and ``vectorfield"
by ``$G$-invariant metric", ``$G$-invariant surface'', and ``$G$-equivariant vectorfield".
Finally,
the result when we do not fix the diffeomorphism type of $M$ follows because there
are only countably many closed manifolds up to diffeomorphism.
(Thus $\MM$ will consist of countably many components corresponding to the 
countable number of diffeomorphism types for $M$.)
\end{proof}

If $\pi$ is a $C^k$ Fredholm map (with $k\ge 1$) of Fredholm index $0$ between separable Banach
spaces, then it follows from the Sard-Smale Theorem~\cite{smale-sard} that the image under $\pi$ of a meager
set is meager.  The following theorem implies (by taking countable unions) a slightly stronger version of that result.

\begin{theorem}\label{key-projection-theorem}
Suppose that $k\ge 1$, that $\MM$ and $\Gamma$ are separable $C^k$ Banach manifolds,
and that $\pi: \MM\to\Gamma$ is a $C^1$ Fredholm map of Fredholm index $0$.
Suppose that $\Ss$ is a closed subset of $\MM$ and that $\Ss\setminus C_\MM$ is nowhere
dense in $\MM$, where $C_\MM\subset \MM$ is the set of critical points of $\pi$.
Then $\pi(\Ss)$ is a meager subset of~$\Gamma$.
\end{theorem}

\begin{proof}
By separability and by definition of ``regular point", we can cover the set of regular points
of $\pi$ by countably many open sets $U_i$ (with $C^k$ boundaries) 
such that $\pi$ maps each $\overline{U_i}$
diffeomorphically onto its (closed) image.
 Now
\begin{equation}\label{union}
\begin{aligned}
\pi(\Ss)
&= 
\pi(\Ss\setminus C_\MM) \cup \pi(\Ss\cap C_\MM)
\\
&= 
\cup_i \pi(\Ss\cap \overline{U_i}) \cup \pi(\Ss\cap C_\MM)
\\
&\subset
\cup_i \pi(\Ss\cap \overline{U_i}) \cup \{\text{critical values of $\pi$} \}.
\end{aligned}
\end{equation}
Each $\Ss\cap \overline{U_i}$ is closed and nowhere dense, so $\pi(\Ss\cap \overline{U_i})$
is closed and nowhere dense.  Also, the set of critical values of $\pi$ is a meager subset
of $\Gamma$ by the Sard-Smale Theorem~\cite{smale-sard}.
Hence  $\pi(\Ss)$ is meager in $\Gamma$ by~\eqref{union}.
\end{proof}

Fix an integer $k\ge 3$, and let $\MM_k$ be as in the Structure Theorem~\ref{structure-theorem}.

\begin{definition}
For positive integers $p$, let $\Ss^p= \Ss^p_k$ be the set of $(\gamma,M)\in \MM_k$
for which the following holds: there is a $p$-sheeted covering $M'$ of a connected
component of $M$ such that $M'$ has a nontrivial jacobi field.
\end{definition}

The following lemma states that $\Ss^p$ satisfies the hypotheses of 
 theorem~\ref{key-projection-theorem}:

\begin{lemma}\label{main-lemma}
The set $\Ss^p=\Ss^p_k$ is a closed subset of $\MM=\MM_k$,
and $\Ss^p\setminus C_\MM$ is nowhere dense in $\MM$.
\end{lemma}

\begin{proof}
Clearly $\Ss^p$ is a closed subset of $\MM$.

To show that $\Ss^p\setminus C_\MM$ is nowhere dense, 
we show that the closure of $\MM\setminus \Ss^p$ contains all of $\Ss^p\setminus C_\MM$.

{\bf Step 1}: The closure of $\MM\setminus \Ss^p$ contains all $(\gamma, M)\in \MM$
such that $\gamma$ is smooth (i.e, $C^\infty$).

Proof of step 1:  Let $(\gamma, M)\in \MM$ be such that $\gamma$ is smooth, which implies that
 $M$
is also smooth.

Let $X$ be the set of $p$-sheeted coverings of connected components of $M$.
Note that $X$ is a finite set:
\[
   X = \{M_1, M_2, \dots, M_n\}.
\]
(Finiteness of $X$ follows from the fact~\cite{hall}*{\S2.4} that a finitely generated group
has at most finitely many subgroups of any given finite index.)

Let $U\subset N$ be a small open set such that $M\cap U$ is a nonempty,
embedded submanifold of $N\cap U$.
Let $f:U\to [0,\infty)$ be a smooth function that is compactly supported in $U$
such that 
\[
   \text{$f=0$ and $Df=0$ on $M\cap U$}
\]
and such that 
\[
   D^2f(x)\nu >0
\]
for some point $x\in M\cap U$ and some normal vector $\nu$ at $x$.
Extend $f$ to all of $N$ by letting $f\equiv 0$ on $N\setminus U$.

Now let $F$ be obtained from $f$ by summing over the group $G$:
\[
    F(y) = \sum_{g\in G} f(g(y)).
\]
Finally, for $t\ge 0$, let $\gamma_t$ be the Riemannian metric
\[
   \gamma_t = (1 + t F) \gamma.
\]

By choice of $F$, we see that  $F=0$ and $DF=0$ on $M$.
It follows that $M$ is minimal with respect to the metric $\gamma_t$, and thus that
\[
  (\gamma_t,M)\in \MM_k.
\]

Let $\lambda_i(M_j,t)$ be the $i^{\rm th}$ eigenvalue of the jacobi operator of $M_j$ with
respect to the metric $\gamma_t$.  Note that $\lambda_i(M_j,t)$ depends continuously on $t$,
and in fact is a strictly increasing function of $t$.  (This is because (1) an eigenfunction
cannot vanish on any nonempty open set, and (2) if a function on $M_j$ does not vanish
on any nonempty open set, then its Rayleigh quotient strictly increases as $t$ increases.)

It follows that there is an $\eps>0$ such that for $t\in (0,\eps)$, none of the $\lambda_i(M_j,t)$
vanish.

Hence (for such $t$) $(\gamma_t,M)\in \MM\setminus \Ss^p$, so 
  $(\gamma, M) \in \overline{\MM\setminus \Ss^p}$.
  This completes the proof of step 1.

{\bf Step 2}: $\overline{\MM\setminus \Ss^p}$ contains all of $\MM\setminus C_\MM$.

{\bf Proof}: Let $(\gamma, M)\in \MM\setminus C_\MM$.
Let $\gamma_i$ be a sequence of smooth (i.e., $C^\infty$) metrics in $\Gamma=\Gamma_k$ 
such that $\gamma_i\to \gamma$
in $C^k$.  Since $(\gamma,M)$ is a regular point of $\pi$, it follows that there
exist (for all sufficiently large $i$) 
surfaces\footnote{these surfaces are unrelated to the $M_i$ in step 1.} $M_i$ such that $(\gamma_i,M_i)\in \MM$ 
and such that $(\gamma_i, M_i)\to (\gamma, M)$.

By step 1, $\overline{\MM \setminus \Ss^p}$ contains each of the $(\gamma_i, M_i)$,
and thus it also contains their limit $(\gamma, M)$.
This completes the proof of step 2 and therefore the proof of lemma~\ref{main-lemma}.
\end{proof}

\begin{remark}
The reader may wonder why it was necessary to break the proof into two steps.
The problem is that if $M$ is a minimal with respect to a $C^k$ metric, then $M$ need not
be $C^k$. (It is $C^{k-1,\alpha}$ for every $\alpha\in (0,1)$.)
If $M$ is not $C^k$, then there is no $C^k$ function $f$ having the properties required in step 1.
\end{remark}

\begin{corollary}\label{key-corollary}
The set $\pi(\cup_p \Ss^p)$ is a meager subset of $\Gamma_k$.
\end{corollary}

\begin{proof}
By theorem~\ref{key-projection-theorem} and lemma~\ref{main-lemma}, 
   $\pi(\Ss^p)$ is a meager subset of $\Gamma_k$.
Thus $\pi(\cup_p\Ss^p)=\cup_p \pi(\Ss^p)$ is also a meager subset of $\Gamma_k$.
\end{proof}

We can now prove the main theorem for $3\le k <\infty$:

\begin{theorem}\label{finite-case-theorem}
Let $k$ be an integer such that $k\ge 3$.
For a generic $G$-invariant metric $\gamma$ in $\Gamma_k$ the following holds:
if $\Sigma$ is a closed, immersed, $\gamma$-minimal submanifold of $N$, then 
$\Sigma$ has no nontrivial jacobi fields.
\end{theorem}

\begin{proof}
It suffices to prove the theorem when $\Sigma$ is connected.
Suppose $\Sigma$ is a closed, connected, immersed, $\gamma$-minimal surface and that 
 $\Sigma$ does have a nontrivial jacobi field.
Let $\hat{M}$ be the union (counting multiplicities) of the  images of $\Sigma$
under the elements of $G$.

By connectivity of $\Sigma$ and by unique continuation, the multiplicity of $\hat{M}$
is equal almost everywhere to some constant $m$.  Let $M$ be obtained from $\hat{M}$ by 
letting the multiplicity
be $1$ almost everywhere.  
Thus $M$ is almost embedded 
and $\hat{M}$ may be regarded as an $m$-sheeted covering of $M$.

Note that $(\gamma,M)\in \MM$ and that $\Sigma$ is a finite-sheeted covering of a connected
component of $M$.
Thus (in the notation of lemma~\ref{main-lemma}),
\[
    (\gamma,M) \in \cup_p\Ss^p
\]
so
\[
  \gamma\in \pi(\cup_p\Ss^p).
\]
The result follows immediately from corollary~\ref{key-corollary}.
\end{proof}

The $k=\infty$ case of theorem~\ref{main-theorem} follows from the $k<\infty$ case 
 (theorem~\ref{finite-case-theorem})
and the following general principle about $C^k$ and $C^\infty$:

\begin{theorem}\label{to-infinity-theorem}
Let $k$ be a finite natural number.
\begin{enumerate}
\item If $U$ is an open dense subset of $\Gamma_k$, then $U\cap \Gamma_\infty$ is an open
dense subset of $\Gamma_\infty$.
\item If $U$ is the intersection of a countable collection of open dense subsets of $\Gamma_k$,
then $U\cap \Gamma_\infty$ is the intersection of a countable collection of open dense subsets
of $\Gamma_\infty$.
\item If $W$ is a second-category subset of $\Gamma_k$ (i.e., if $W$ contains the intersection
of a countable collection of open dense subset of $\Gamma_k$), then $W\cap \Gamma_\infty$
is a second-category subset of $\Gamma_\infty$.
\end{enumerate}
\end{theorem}

\begin{proof} The first assertion is trivially true, the second follows immediately from the first,
and the third follows immediately from the second.
\end{proof}

\begin{proof}[Proof of theorem~\ref{main-theorem} for $k=\infty$]
Let $W_k$ be the set of $G$-invariant, $C^k$ metrics on $N$ for with
the following property: if $M$ is a closed, minimal immersed submanifold  of $N$,
then $M$ has no nontrivial jacobi fields.

By theorem~\ref{finite-case-theorem}, $W_3$ is a second-category subset of $\Gamma_3$.
Note that
\[
     W_\infty = W_3\cap \Gamma_\infty.
\]
Thus by theorem~\ref{to-infinity-theorem}, $W_\infty$ is a second-category subset 
of $\Gamma_\infty$.
\end{proof}

\section{Generalizations}\label{generalizations-section}

\subsection*{Conformal classes}
The main theorem, Theorem~\ref{main-theorem}, remains true if we replace $\Gamma_k$ by the set of $G$-invariant, $C^k$ metrics
in the conformal class of some specified $G$-invariant metric.  No changes are required in the
proof.

\subsection*{Partially fixed metrics}
Let $\gamma_0$ be a $G$-invariant, $C^k$ metric on $N$.
Let $U$ be a $G$-invariant open subset of $N$.

\begin{theorem}
Suppose $3\le k\le \infty$ and $j\le k$. 
Consider the class $\Gamma^*$ of $G$-invariant, $C^k$ metrics $\gamma$ on $N$ such that
$\gamma$ and $\gamma_0$ and their derivatives of order $\le j$ agree on $N\setminus U$.
Then a generic metric $\gamma$ in $\Gamma^*$ has the following property:
if $M$ is a closed, minimal immersed submanifold of $N$ and if each connected component
of $M$ intersects $U$, 
then $M$ has no nontrivial jacobi fields.
\end{theorem}

The proof is almost exactly the same as the proof of Theorem~\ref{main-theorem}.
As with theorem~\ref{main-theorem}, the theorem remains true (with the same proof)
 if we replace $\Gamma^*$ by the set of metrics
in $\Gamma^*$ that are conformally equivalent to $\gamma_0$.

\section{An open problem}

In general, properties of closed minimal surfaces that are generic
with respect to varying metrics are also generic for minimal surfaces with boundary
with respect to variations of the boundary, and vice versa.
For example, for generic metrics, closed minimal surfaces have no nontrivial
jacobi fields, and, likewise, for generic boundary curves (with a fixed ambient metric), 
the minimal surfaces they bound have
no nontrivial jacobi fields that vanish at the boundary.

Thus Theorem~\ref{main-theorem} suggests the following conjecture:

\begin{conjecture}
Let $G$ be a finite set of isometries of $\RR^n$.
For a generic smooth,  closed, $G$-invariant, embedded submanifold of $\RR^n$,
none of the the smooth immersed minimal submanifolds it bounds have nontrivial jacobi
fields that vanish at the boundary.
\end{conjecture}

If $G$ is the trivial group, this is true~\cite{white-space1}*{3.3(5)}.
More generally, the conjecture is true if we replace ``minimal submanifold'' by $G$-invariant minimal submanifold"
and ``jacobi field" by ``$G$-invariant
jacobi field".   However, I do not know how to prove the conjecture as stated.
Of course, the conjecture should remain true if $\RR^n$ is replaced by any
smooth Riemannian $n$-manifold.

\nocite{pedrosa-ritore}
\nocite{hoffman-wei}
\newcommand{\hide}[1]{}

\end{document}